\documentclass[10pt]{article}
\usepackage[latin1]{inputenc}
\usepackage{amsthm}
\usepackage{amssymb}
\usepackage{amsmath}
\theoremstyle{plain}
\theoremstyle{plain}
\newtheorem{thm}{Theorem}
\newtheorem*{thmA}{Theorem A}
\newtheorem*{thmAsharp}{Theorem A$\sharp$}
\newtheorem*{mainthm}{Main Theorem}
\newtheorem*{propB}{Proposition B}
\newtheorem*{propC}{Proposition C}

\newtheorem*{propD}{Proposition D}

\newtheorem*{thmE}{Theorem E}
\newtheorem*{thmEsharp}{Theorem E$\sharp$}
\theoremstyle{definition}
\newtheorem{defn}[thm]{Definition}
\theoremstyle{plain}
\newtheorem{prop}[thm]{Proposition}

\newtheorem{lem}[thm]{Lemma}

\theoremstyle{remark}
\newtheorem{rem}[thm]{Remark}
\newtheorem{ex}[thm]{Example}

\usepackage{graphicx}
\usepackage{amscd}
\usepackage[all,knot,poly]{xy}
\usepackage[english]{babel}
\usepackage{mathrsfs}
\usepackage{hyperref}
\newcommand{\R}{\mathbb{R}}
\newcommand{\Q}{\mathbb{Q}}
\newcommand{\Z}{\mathbb{Z}}
\newcommand{\N}{\mathbb{N}}

\newcommand{\OO}{\mathcal{O}}
\newcommand{\F}{\mathcal{F}}
\newcommand{\G}{\mathcal{G}}
\newcommand{\M}{\mathcal{M}}
\newcommand{\all}{\mathcal{A}ll}

\newcommand{\stab}{\mathrm{stab}}
\newcommand{\id}{\trianglelefteq}

\newcommand{\spn}{\mathrm{span}}

\newcommand{\FP}{\mathrm{FP}}
\newcommand{\gen}[1]{\langle #1\rangle}

\newcommand{\slot}{\,\underbar{\;\;}\,}

\author{Federico William Pasini}
\title{The ambient classifying space of a classical knot group}
\date{\today}
\begin{document}
\maketitle
\begin{abstract}
For a prime knot group, the classifying space for the family of the subgroups generated by the meridians can be seen as an abstract analogue of the ambient manifold in which the knot lives.
An explicit model of this \textit{ambient} classifying space is constructed as a branched covering space of the $3$-sphere branched over the knot; more general branched covering spaces, obtained quotienting the ambient classifying space by finite-index normal subgroups, are studied.
Various homological properties of said spaces are established, some of which have parallels in algebraic number theory. In particular, prime knot groups are shown to be Bieri-Eckmann, but not Poincar\'e, duality groups for Bredon cohomology with respect to the family of the meridians.
\end{abstract}

\section{Introduction}
Classifying spaces $E_\F G$ of a group $G$ for families $\F$ of subgroups are a generalisation of the total space $EG$ of the universal $G$-principal bundle $EG\to BG$. They have recently attracted much attention, with a special focus on the families $\F=\mathcal{F}in$ of finite subgroups and $\F=\mathcal{V}cyc$ of virtually cyclic subgroups, in connection with Baum-Connes and Farrell-Jones Conjectures respectively (cf.~\cite{luecksurvey}, \cite{misval} and \cite{farjon}).

In knot theory, a different family of the knot group $G_K$ associated to a (tame) knot $K$ is naturally available. It is the \emph{family of meridians}, that is the family $\M$ of the infinite cyclic subgroups generated by a power of a generator of $G_K$ (cf.~\S \ref{sec:construction}). The classifying space associated to $\M$, which we call the \emph{ambient classifying space} of the knot group, turns out to have many significant properties, especially in the case of prime knots.

Finding friendly models of classifying spaces for families is a major problem in geometric group theory, and usually a tough nut to crack.
In our situation, a concrete model of the classifying space of a prime knot group $G_K$ for the family $\M$ can be obtained adapting a technique developed by L\"uck and Weiermann (cf.~\cite{luewei2007}).
\begin{mainthm}
 Let $G_K$ be a prime knot group and $H=\gen{a,l}\cong\Z^2$ a peripheral subgroup generated by a meridian $a$ and a longitude $l$. Then [a model of] the classifying space $E_\M G_K$ of $G_K$ for the family of meridians is given by the pushout of $G_K$-CW-complexes
\begin{equation*}
\xymatrix{
\bigsqcup_{G_K/H}\R^2\ar^\phi[rrr]\ar^{id_{G_K}\times\psi}[d] &&& EG_K\ar[d]\\
\bigsqcup_{G_K/H}Cyl(\pi\colon\R^2\to\R)\ar[rrr]&&&E_\M G_K.
}
\end{equation*}
Here $Cyl(\pi\colon\R^2\to\R)$ denotes the mapping cylinder of the projection from $EH$ to the subcomplex $EH^{\gen a}$ of points fixed by the subgroup $\gen a$. The maps starting from the upper-left corner glue the term $EH\approx\R^2$ in each copy of that mapping cylinder to a boundary component of $EG_K$.
\end{mainthm}

Note that the model of $E_\M G_K$ provided by the Main Theorem includes the chosen model of $EG_K$ as a sub-$G_K$-CW-complex. Also, the natural model of $BG_K$ is a deformation retract of the knot complement $S^3\setminus K$ (cf.~\S \ref{ssec:preliminaries}).

The same construction works for non-prime knots as well. Yet, the maps in the resulting pushout, and hence the model of the classifying space, are less explicit (\S \ref{ssec:non-prime knots}). The technique can also be extended to links (\S \ref{ssec:hopf}). The proof of the Main Theorem and its extensions occupies Section \ref{sec:construction}.

\subsection{Coverings}

The name \emph{ambient classifying space} for the space $E_\M G_K$ is justified by the following direct corollary of the Main Theorem.
\begin{thmA}
Let $K$ be a prime knot. Then $E_\M G_K/G_K\simeq S^3$.
\end{thmA}

Any normal subgroup of finite index $U\id G_K$ determines a unique finite regular covering $\pi^U\colon EG_K/U\to BG_K$, characterised by $\pi_1(EG_K/U, x)\cong U$ (cf.~\cite[\S 1.3]{hatcher}). The Main Theorem shows that $U$ also determines an associated finite \emph{branched} covering $\widehat\pi^U\colon E_\M G_K/U\to S^3$, branched over the knot.

The subgroup $U$ has a natural family of subgroups obtained by restricting $\M$:
\begin{equation}\label{eq:M_U}
\M_U=\{H\cap U\mid H\in\M\}=\{H\leq U\mid H\in\M\}.
\end{equation}

Theorem A extends to normal subgroups of finite index, with respect to their restricted family:
\begin{thmAsharp}
 Let $K$ be a prime knot, $G_K$ its knot group and $U\id G_K$ a normal subgroup of finite index. Then the lift $(\pi^U)^{-1}(K)$ is a link living in the $3$-manifold $E_{\M_U} U/U=E_\M G_K/U$.
\end{thmAsharp}

Moreover, the unbranched and branched covering spaces associated to $U$ are related to each other at the level of their fundamental groups:
\begin{propB}
Let $G_K$ be a prime knot group and $U\id G_K$ a normal subgroup of finite index. There is a short exact sequence of groups
\begin{equation}\label{eq:non ab}
\xymatrix{
1\ar[r]&M_U\ar[r]&U\ar[r]&\pi_1(E_\M G_K/U)\ar[r]&1,
}
\end{equation}
where $M_U$ is the normal subgroup of $U$ generated by those powers of the meridians of the knot that lie in $U$. In particular, applying the abelianisation functor one gets the exact sequence
\begin{equation}\label{eq:ab}
\xymatrix{
 M_U^{ab}\ar[r]&U^{ab}\ar^-{\pi}[r]&H_1(E_\M G_K/U)\ar[r]&0.
 }
\end{equation}
\end{propB}
\begin{rem}
This has a nice parallel in algebraic number theory. 
Let $F$ be an algebraic number field, $\mathcal O_F$ its ring of algebraic integers and $G_F$ its absolute Galois group, that is the Galois group of its separable closure (cf.~\cite[\S IV.1]{neukirch}).
A direct consequence of Hilbert class field theory is that the ideal class group $Cl_F$ (cf.~\cite[\S I.3]{neukirch}) fits into an exact sequence
\[
 \xymatrix{
 \coprod_{p\in \mathrm{Spec(\OO_F)}}\OO_p^*\ar[r]&G_F^{ab}\ar^-{\varpi}[r]&Cl_F\ar[r]&0,
 }
\]
where $\OO_p$ is the complete discrete valuation ring associated with the prime $p\id\OO_F$ and $G_F^{ab}$ is the Galois group of the maximal abelian extension of $F$.
Even if an explicit description has not been obtained yet, it is known that $\ker(\varpi)$ is generated by the \emph{ramification groups} (cf.~\cite[\S II.9]{neukirch}) of the primes of $\OO_F$. So, passing from the maximal abelian extension of $F$ to the extension of $F$ with Galois group $Cl_F$ results in killing algebraic ramification.

On the other hand, a way of reading the sequences \eqref{eq:non ab} and \eqref{eq:ab} is to say that the ``meridians'' of $U$, i.e.~the generators of $M_U$, are responsible for the ramification in the branched covering space $E_\M U$. Then,
passing from an unbranched covering space to the associated branched one results in killing topological ramification.  
This analogy between knot theory and algebraic number theory is in the spirit of \cite{morishita}.
\end{rem}

Proposition B is a key step in proving various statements about branched coverings of $S^3$ branched over a knot (cf.~\S \ref{sec:consequences}).
The significance of this topic is well expressed, for example, by J. W. Alexander's theorem that every compact orientable $3$-manifold without boundary is a branched covering space of $S^3$ branched over a link (cf.~\cite{alexander}). More recently, H. Hilden, M. T. Lozano and J. M. Montesinos improved the result, showing that
every compact orientable $3$-manifold without boundary is a branched covering space of $S^3$ branched over the \emph{figure-eight knot} (cf.~\cite{hilomo}).

\subsection{Poitou-Tate duality}
Poincar\'e-Lefschetz duality yields the following.

\begin{propC}
Let $G_K$ be a nontrivial knot group, $H\leq G_K$ a peripheral subgroup and $U\id G_K$ a normal subgroup of finite index. Then there is an exact sequence of groups
\begin{equation}\label{eq:poitou-tate U}
\xymatrix@R-10pt{
0\ar[r] & H^0(U,\Z)\ar[r] & \coprod\limits_{G_K/U\! H} H_2(H\cap U,\Z)\ar[r] & H_2(U,\Z)\ar[lld] & \\
 & H^1(U,\Z)\ar[r] & \coprod\limits_{G_K/U\! H} H_1(H\cap U,\Z)\ar[r] & H_1(U,\Z)\ar[lld] & \\
 & H^2(U,\Z)\ar[r] & \coprod\limits_{G_K/U\! H} H_0(H\cap U,\Z)\ar[r] & H_0(U,\Z)\ar[r] & 0.
}
\end{equation}
\end{propC}
Note that $H\cap U$ is isomorphic to $\Z^2$, as $U$ has finite index in $G_K$. 
\begin{rem}
This has an analogue in algebraic number theory as well.
In few words, there is a duality principle, named after Poitou and Tate, for Galois cohomology of algebraic number fields. If $S$ is a nonempty set of primes of an algebraic number field $F$, including the primes at $\infty$, and $A$ is a finite $G_S$-module s.t. $\nu(|A|)=0$ for all $\nu\notin S$, Poitou-Tate duality produces the $9$-terms exact sequence
\begin{equation}\label{eq:poitou-tate F}
\xymatrix@R-10pt{
0\ar[r] & H^0(F_S,A)\ar[r] & P^0(F_S,A)\ar[r] & H^2(F_S,A')^\vee\ar[dll]&\\
& H^1(F_S,A)\ar_{loc_1}[r] & P^1(F_S,A)\ar[r] & H^1(F_S,A')^\vee\ar[dll]&\\
& H^2(F_S,A)\ar[r] & P^2(F_S,A)\ar[r] & H^0(F_S,A')^\vee\ar[r]&0.
}
\end{equation}
This connects the Galois cohomology of the maximal extension $F_S/F$ unramified outside $S$ with the Galois cohomology of its completions with respect to a set of non-archimedean distances induced by the prime ideals of the ring of algebraic integers of $F_S$ (more details can be found in \cite[\S VIII.6]{NSW} or \cite[\S II.6]{serre}).


\end{rem}

The kernel of the \emph{localisation map} $loc_1$ in Poitou-Tate sequence \eqref{eq:poitou-tate F} is the \emph{first Tate-\v{S}afarevi\v{c} group} of $F_S$ (cf.~\cite[\S VIII.6]{NSW}). Under certain circumstances, the first homology and cohomology groups of $E_\M G_K/U$ play an analogue role in the sequence \eqref{eq:poitou-tate U}.
\begin{propD}
Let $G_K$ be a prime knot group, $H\leq G_K$ a peripheral subgroup and $U\id G_K$ a normal subgroup of finite index. If the link $(\widehat\pi^U)^{-1}(K)$ in the $3$-manifold $E_\M G_K/U$ is a knot and its longitude is null-homologous in its complement $E_\M G_K/U\setminus(\widehat\pi^U)^{-1}(K)$, then there is an exact sequence of groups
\[
\xymatrix@R-10pt@C-8pt{
0\ar[r] & H^1(E_\M G_K/U,\Z)\ar[d]\\
 & H^1(U,\Z)\ar[r] & \coprod\limits_{G_K/U\! H} H_1(H\cap U,\Z)\ar[r] & H_1(U,\Z)\ar[d] & \\
 &&&H_1(E_\M G_K/U,\Z)\ar[r] & 0.
}
\]
\end{propD}

Besides the charming analogies, Proposition D highlights also a remarkable difference between knots and numbers. That is, in the realm of knots the first homology and the first cohomology group of the same space $E_\M G_K/U$ appear as candidates to complete the degree-$1$ subsequence of \eqref{eq:poitou-tate U} on both ends. This has no (known) analogue in the Poitou-Tate setting, which then is more difficult to handle.

An illustrative example is Leopoldt Conjecture in number theory. Let $F$ be an algebraic number field with $[F:\Q]=n$. Suppose for simplicity that $F\supseteq\Q[i]$, so that $F$ has no real embeddings, and that $F$ contains a primitive $p$th root of unity, for some prime $p\in\Z$. Set $S_p=\{Q\id\OO_F\mid Q\cap\Z= p\}$, $S_\infty$ the set of infinite places of $F$ (cf.~\cite[\S III.1]{neukirch}) and $S=S_p\cup S_\infty$ (note that this is a finite set). Furthermore, let $D_Q$ be the decomposition group of $Q$ (cf.~\cite[\S I.9]{neukirch}) and let $G_F^S$ be the Galois group of the maximal field extension of $F$ unramified outside $S$ (cf.~\cite[\S III.2]{neukirch}).
Then there is an exact sequence
\[
\xymatrix{
 H^1\left(G_F^S,\Z_p(1)\right)\ar[r] & \coprod_{Q\in S}H_1\left(D_Q,\Z_p\right)\ar[r] & H_1\left(G_F^S,\Z_p\right),
 }
\]
where $\Z_p(1)$ is the \emph{Tate twist} (cf.~\cite[\S VII.3]{neukirch}).
The Leopoldt Conjecture claims that if one completes this sequence to
\[
\xymatrix@R-10pt{
0\ar[r] & T_1\ar[d]&&&\\
& H^1\left(G_F^S,\Z_p(1)\right)\ar[r] & \coprod_{Q\in S}H_1\left(D_Q,\Z_p\right)\ar[r] & H_1\left(G_F^S,\Z_p\right)\ar[d]&\\
&&& T_2\ar[r] & 0
 }
\]
then the groups $T_1$ and $T_2$ are torsion groups. There are cases, for example when one passes to the maximal pro-$p$ quotient $G_F^S(p)$ of $G_F^S$, in which $T_2$ is known to be even finite, but this still does not give any hints on $T_1$. On the contrary, in the assumptions of Proposition D, the finiteness of $H_1(E_\M G_K/U,\Z)$ would immediately imply the triviality of $H^1(E_\M G_K/U,\Z)$, via the Universal Coefficient Theorem for cohomology (cf.~\cite[\S 3.1]{hatcher}).

R.~Bieri and B.~Eckmann studied a notion of homological duality for groups that generalises Poincar\'e duality (cf.~\cite{bieck}). Nontrivial knot groups are duality groups in this sense.

In my Ph.D.~thesis \cite{fitz}, Bieri-Eckmann duality of knot groups is used to prove Propositions C and D through the machinery of derived categories with duality.

\subsection{Equivariant duality}
The natural choice of a (co)homology theory for equivariant CW-complexes with stabilisers in a prescribed family of subgroups is Bredon (co)homology (cf.~\cite{bredon}, \cite{lueckbredon}).
M.~Davis and I.~Leary (cf.~\cite[\S 1]{davlea}) and C.~Mart\'\i nez-P\'erez (cf.~\cite[Definition 5.1]{martinez} have extended the notions of Poincar\'e duality and (Bieri-Eckmann) duality to Bredon cohomology with respect to the family of finite subgroups. However, their generalisations actually make sense for Bredon cohomology with respect to any family of subgroups.
\begin{defn}\label{def:bredon duality}
Let $\F$ be a family of subgroups of a group $G$ and let $R$ be a ring.
$G$ is Bredon $\FP$ for the family $\F$ (in short, $\F-\FP$) if the trivial $\F$-Bredon module $\Z_\F$ has a finite-length resolution by finitely generated projective $\F$-Bredon modules (cf.~\cite[Chapter 3]{fluch}, \cite[\S 3]{kromarnuc}).

$G$ is a \emph{$\F$-Bredon duality group} over $R$ if it is $\F-\FP$ and for every $H\in\F$ there is a positive integer $n_H$ such that
\begin{equation}\label{eq:bredon duality}
H^i(WH,R[WH])=\begin{cases}
                     (R\mbox{-flat module}) & \mbox{if }i=n_H\\
                     0 & \mbox{otherwise.}
                     \end{cases}
\end{equation}
Here $WH=N_G H/H$ is the \emph{Weyl group} of $H$.
In particular, $G$ is a \emph{$\F$-Bredon Poincar\'e duality group} over $R$ if, for all $H\in \F$, $H^{n_H}(WH,R[WH])=R$.
\end{defn}

The $\F-\FP$ condition can be verified topologically (cf.~\cite[Theorem 0.1]{luemei}), exactly as for the classical $FP$ condition (cf.~\cite[\S VIII.6]{brown}). Detailed discussions of Bredon cohomological finiteness conditions (for arbitrary families of subgroups) and Bredon (Poincar\'e) duality groups (for the family of finite subgroups) can be found in \cite{fluch} and \cite{SJGarticle, stjohngreen}.

\begin{thmE}
Let $G_K$ be a prime knot group. Then $G_K$ is a $\M$-Bredon duality group, but not a $\M$-Bredon Poincar\'e duality group, over $\Z$.
\end{thmE}
The property extends to all normal subgroups of finite index:
\begin{thmEsharp}
Let $G_K$ be a prime knot group and $U$ a normal subgroup of finite index. Then $U$ is a $\M_U$-Bredon duality group, but not a $\M_U$-Bredon Poincar\'e duality group, over $\Z$.
\end{thmEsharp}

%
%
%
%
%

\section{The construction of ambient classifying spaces}\label{sec:construction}

\subsection{Preliminaries}\label{ssec:preliminaries}
Let $G$ be a group. A collection $\F$ of subgroups of $G$ is a \emph{family} if it is closed under taking subgroups and under conjugation. 
A \emph{[model of the] classifying space} $E_\F(G)$ \emph{of the family} $\F$ is a $G$-CW-complex $X$ such that for all subgroups $H\leq G$ the fixed-point set 
\(
 X^H=\left\{ x\in X\mid\forall h\in H:h(x)=x\right\}
\)
is contractible if $H\in\F$ and empty otherwise.

Classifying spaces are known to exist and to be unique up to $G$-homotopy for all families of all groups $G$ (cf.~\cite[Subsection 1.2]{luecksurvey}). Yet, the construction provided by J. Milnor to show general existence (cf.~\cite{milnorI, milnorII}) usually produces very complicated spaces to deal with. A major problem is to find other models that are as small and nice as possible, on which to carry effective computations.

A result by W. L\"uck and M. Weiermann (\cite[Theorem 2.3]{luewei2007}) explains how to build classifying spaces for bigger families from those for smaller families. We recall how their machinery works.
Let $\F\subseteq\G$ be two families of the group $G$. An equivalence relation $\sim$ on $\G\setminus\F$ with the additional properties
\begin{eqnarray}
H,K\in \G\setminus\F, H\subseteq K &\Longrightarrow & H\sim K\label{tilde1}\\
\label{tilde2} H,K\in \G\setminus\F, H\sim K &\Longrightarrow &\forall g\in G: g^{-1}Hg\sim g^{-1}Kg
\end{eqnarray}
will be called a \emph{strong equivalence relation}.
Given such $\sim$, we denote $[\G\setminus \F]$ the set of $\sim$-equivalence classes and $[H]$ the $\sim$-equivalence class of $H$ and define the subgroup of $G$
\[
N_G[H]=\{g\in G\mid [g^{-1}Hg]=[H]\}.
\]
Property (\ref{tilde2}) says that the conjugation of subgroups induces an action on $[\G\setminus \F]$ with respect to which $N_G[H]$ is the stabiliser of $[H]$.
Finally, construct the family
\[
\G[H]=\{K\leq N_G[H]\mid K\in\G\setminus \F, [K]=[H]\}\cup \{K\leq N_G[H]\mid K\in\F\}
\]
of subgroups of $N_G[H]$.
\begin{thm}[L\"uck \& Weiermann]\label{thm:lueckweiermann} Let $I$ be a complete system of representatives $[H]$ of the $G$-orbits in $[\G\setminus \F]$ under the $G$-action coming from conjugation. Choose arbitrary $N_G[H]$-CW-models for $E_{\F\cap N_G[H]}N_G[H]$ and $E_{\G[H]}(N_G[H])$,
and an arbitrary $G$-CW-model for $E_\F G$. Define a $G$-CW-complex $Y$ by the
cellular $G$-pushout
\begin{equation}\label{eq:pushout diagram}
\xymatrix{
\bigsqcup_{[H]\in I}G\times_{N_G[H]}E_{\F\cap N_G[H]}(N_G[H])\ar^-\phi[rrr]\ar^{\bigsqcup_{[H]\in I}id_G\times\psi_{[H]}}[d]&&& E_\F G\ar[d]\\
\bigsqcup_{[H]\in I}G\times_{N_G[H]}E_{\G[H]}(N_G[H])\ar[rrr]&&&Y
}
\end{equation}
such that the $G$-map $\phi$ is cellular and all the $N_G[H]$-maps $\psi_{[H]}$ are inclusions or viceversa.

Then $Y$ is a model for $E_\G G$.
\end{thm}

Let $K\simeq S^1\subseteq S^3$ be a (tame) knot. A \emph{tubular neighbourhood} $V_K$ of $K$ is an open solid torus centered at $K$ and thin enough to have no self-intersections. The \emph{knot group} of $K$ is $G_K=\pi_1(S^3\setminus K,x)\cong\pi_1(S^3\setminus V_K,x)$. The space $X_K=S^3\setminus V_K$ is a compact $3$-manifold with boundary and has the structure of a CW-complex. It will be called the \emph{closed knot complement}. The fundamental group of the boundary $T_K=\partial V_K$ is isomorphic to $\Z^2$. It is generated by a \emph{meridian} $m$ and a \emph{longitude} $l$, characterised as in \cite[\S 3.A]{burzie}. We always choose a base point $x\in T$ when computing fundamental groups. If $K$ is nontrivial, that is, not ambient-isotopic to the standard circle, then the inclusion-induced homomorphism $\iota_\ast\colon\pi_1(T_K,x)\to\pi_1(X_K,x)$ is injective. In particular, $\iota_\ast(m)$ is a generator, say $a$, of $G_K$ and $H_a:=\iota_\ast(\pi_1(T_K,x))\cong\Z^2$ is the \emph{peripheral subgroup} containing $a$. All the other generators of $G_K$ are known to be conjugate to either $a$ or $a^{-1}$. Hence it is possible to obtain peripheral subgroups containing each of them by varying the basepoint $x$. In view of this, with a little abuse we call \emph{meridians} of a knot all the conjugates of the generators of its knot group. The \emph{family of meridians} is the collection
\[
\M=\left\{\gen{g^{-1}a^ig}\mid g\in G_K, i\in\N\right\}
\]
of infinite cyclic subgroups generated by each power of a meridian.

The above also makes sense for (tame) links. One has only to choose pairwise separated tubular neighbourhoods, one for each component. The big difference is that the generators of a link group need not be all conjugate to a fixed one or its inverse, so in general the family $\M$ has a more complicated description. 

As a consequence of C.~Papakyriakopoulos's Sphere Theorem (cf.~\cite{papa}), the closed complement $X_K$ of a knot $K$ in $S^3$ is aspherical. Hence, it is a model of the Eilenberg-MacLane space $BG_K=K(G_K,1)$ (cf.~\cite[\S 1.B]{hatcher}. Its universal cover $\widetilde{X_K}$ is then a model of the total space $E G_K$ of the universal principal bundle $EG_K\to BG_K$.

The space $\widetilde{X_K}$ can be also viewed as a model of the classfying space of $G_K$ for the trivial family $\{\{1\}\}$. We can then apply L\"uck and Weiermann's Theorem setting $G$ to be the group of the knot $K$, $\F$ to be the trivial family, $\G$ to be the family $\M$ of meridians.

If the knot $K$ is trivial, then $\M$ is the set of all subgroups of $G_K\simeq\Z$, so $E_\M G_K$ is a point. This degenerate situation having been settled, from now on we assume $K$ to be nontrivial.

\begin{rem}\label{rem:GG partially ordered}
$\M$ is a partially ordered set with respect to inclusion, it has maximal elements $\gen{g^{-1}ag}, g\in G_K$ and each element of $\M$ is included in a maximal one.
\end{rem}

\subsection{Definition of $\sim$}
For $H\in\M\setminus\F$, we set $\downarrow\! H:=\{K\in\M\setminus\F\mid K\leq H\}$.
%
\begin{lem}
The relation $\sim$ on $\M\setminus\F$ defined as
\[
H\sim K \Longleftrightarrow \exists L\in\M\setminus\F :L\leq H,L\leq K
\]
is a strong equivalence relation.
\end{lem}
\begin{proof}
The only non obvious property is transitivity. A partially ordered set $(P,\preceq)$ is called \emph{downward-directed} if for all $x,y\in P$ there is a $z\in P$ such that $z\preceq x$ and $z\preceq y$. As all the elements of $\M\setminus\F$ are infinite cyclic, for all $H\in\M\setminus\F$ the set $\downarrow\! H$ is downward-directed with respect to inclusion.

Let then $H_1\sim H_2$ and $H_2\sim H_3$, that is $\exists K_1\in\M\setminus\F :K_1\leq H_1,K_1\leq H_2$ and $\exists K_2\in\M\setminus\F :K_2\leq H_2,K_2\leq H_3$. In particular $K_1, K_2\in\ \downarrow\! H_2$, so that by directedness there exists a $L\in\ \downarrow\! H_2\subseteq\M\setminus\F$ which is included in both. All three $H_i$'s include this $L$, hence $H_1\sim H_3$.
\end{proof}

Recalling Remark \ref{rem:GG partially ordered}, we can find a maximal element in each $\sim$-equivalence class and we choose it as the representative of its class.

\subsection{$N_{G_K}[H]$ and $\M[H]$}
As far as we are concerned with knots, the conjugation action of $G_K$ on $[\M\setminus\F]$ is transitive by the very definition. So the stabilisers are all conjugate to one another and we need to study only $N_{G_K}[\gen{a}]$.

\begin{lem}\label{lem:non-cable normalizer}
If $K$ is a nontrivial knot, then $N_{G_K}[\gen{a}]=N_{G_K}(a)=C_{G_K}(a)$.
\end{lem}
\begin{proof}
The inclusions $N_{G_K}[\gen{a}]\supseteq N_{G_K}(a)\supseteq C_{G_K}(a)$ are obvious.

By the definition of $\sim$, a certain $t\in G_K$ lies in $N_{G_K}[\gen{a}]$ if and only if $\gen{t^{-1}at}\cap \gen{a}\neq\{1\}$, if and only if $\exists i,j\in\Z\setminus\{0\}$ such that $t^{-1}a^it=a^j$. If we project the last equality onto $G_K/[G_K,G_K]$, which is infinite cyclic generated by the image of $a$, we see that $i=j$. Hence $t$ centralises a power of $a$. But, according to \cite[Corollary 3.7]{jacshaperi}, 
the centraliser in $G_K$ of an element of the peripheral subgroup coincides with the centraliser of any power of the given element.
Summing up, all $t\in N_{G_K}[\gen{a}]$ centralise $a$.
\end{proof}
\begin{rem}\label{rem:GG maximal elements}
This also means that Remark \ref{rem:GG partially ordered} can be strengthened: each element of $\M$ is included in a \emph{unique} maximal element.
\end{rem}

The same idea yields:
\begin{lem}\label{lem:non-cable family}
If $K$ is a nontrivial knot, then $\M[\gen{a}]=\all\gen{a}:=\,\downarrow\!\gen{a}\cup\big\{\{1\}\big\}$.
\end{lem}
\begin{proof}
One has just to observe that the simultaneous conditions $M\in\M\setminus\F$ and $[M]=[\gen{a}]$ 
mean $M=\gen{t^{-1}a^kt}$, for some $k\in\N\setminus\{0\}$ and some $t\in G_K$ which, by the same argument as in the previous lemma, lies in $C_{G_K}(a)$. It follows that $M=\gen{a^k}\leq \gen{a}$.
\end{proof}

\subsection{Prime knots}\label{ssec:prime knots}
A nontrivial knot is \emph{prime} if it cannot be obtained as the \emph{knot product} (also called \emph{knot sum} or \emph{composition}) of nontrivial knots (cf.~\cite[\S 2.C, 7.A]{burzie}).

If $K$ is prime, it follows from J. Simon's results \cite{simon1976} that the centraliser of a meridian is the peripheral subgroup containing it, $C_{G_K}(a)=\gen{a,l}\simeq\Z^2$, where $l$ is a longitude corresponding to $a$. So Diagram (\ref{eq:pushout diagram}) becomes
\begin{equation*}
\xymatrix{
G_K\times_{\gen{a,l}}E\Z^2\ar^\phi[rrr]\ar^{id_{G_K}\times\psi}[d] &&& EG_K\ar[d]\\
G_K\times_{\gen{a,l}}E_{\all\gen{a}}\Z^2\ar[rrr]&&&E_\M G_K.
}
\end{equation*}
The top-left space is a disjoint union of copies of $\R^2$ indexed by $G_K/\gen{a,l}$, glued via $\phi$ to the boundary components of the top-right space. This one is a $3$-manifold, with boundary a disjoint union of planes indexed exactly by $G_K/\gen{a,l}$ (so $\phi$ can clearly be chosen to be a cellular map). As for the bottom-left corner, consider $\R^2$, on which $\gen{a,l}$ acts freely by integer translations, as a model for $E(\Z^2)$, and $\R$, on which $l$ acts by integer translations while $a$ fixes every point, as a model for $E(\Z^2)^{\gen{a}}$. Then the space in that corner can be chosen to be the disjoint union, again indexed by $G_K/\gen{a,l}$, of copies of the mapping cylinder $Cyl(\pi\colon\R^2\to\R)$ of the projection $E(\Z^2)\to E(\Z^2)^{\gen{a}}\simeq E(\Z^2)/\gen{a}$.
This choice guarantees the map $\psi$ to be the inclusion of $\R^2$ into the mapping cylinder, as required in L\"uck and Weiermann's Theorem \ref{thm:lueckweiermann}.
Finally, $E_\M G_K$ is the $3$-dimensional $G_K$-CW-complex obtained by gluing each $Cyl(\pi\colon\R^2\to\R)$, along its $\R^2$ copy, to one boundary component of $EG_K$.

Summing up, for a prime knot the classifying space $E_\M G_K$ is given by the pushout
\begin{equation}\label{eq:prime diagram explicit}
\xymatrix{
\bigsqcup_{G_K/H}\R^2\ar^\phi[rrr]\ar^{id_{G_K}\times\psi}[d] &&& EG_K\ar[d]\\
\bigsqcup_{G_K/H}Cyl(\pi\colon\R^2\to\R)\ar[rrr]&&&E_\M G_K.
}
\end{equation}
This concludes the proof of the Main Theorem.

\subsection{Composite knots}\label{ssec:non-prime knots}
A nontrivial knot is \emph{composite} if it is not prime. Any composite knot $L$ admits an essentially unique decomposition into a knot product of prime factors $L=K_1\sharp\dots\sharp K_r$ (cf.~\cite{schubert1949}). According to \cite{noga1967}, the centraliser $C_{G_K'}(a)$ of the meridian $a$ in the commutator subgroup of $G_K$ is free on the longitudes $l_1,\dots,l_r$ of those prime factors.

But then,
\begin{equation*}
\frac{C_{G_K}(a)}{C_{G_K'}(a)}=\frac{C_{G_K}(a)}{C_{G_K}(a)\cap G_K'}\simeq \frac{C_{G_K}(a)G_K'}{G_K'}\leq \frac{G_K}{G_K'}.
\end{equation*}
Since the last group is infinite cyclic, the first one is so as well. More precisely, since $a\in C_{G_K}(a)\setminus G_K'$, a generator of this group is $aC_{G_K'}(a)$.
This amounts to saying that there is a right-split exact sequence
\begin{equation*}
1\rightarrow C_{G_K'}(a)\rightarrow C_{G_K}(a)\rightleftarrows \gen{a}\rightarrow 1
\end{equation*}
or equivalently, since $a$ is central in $C_{G_K'}(a)$,
\begin{equation*}\label{eq:centralizer non-prime non-cable}
C_{G_K}(a)=\gen{a}\times C_{G_K'}(a)\simeq \Z\times F_r
\end{equation*}
(where $F_r$ denotes the free group of rank $r$).

In this setting, Diagram \eqref{eq:pushout diagram} becomes
\begin{equation*}
\xymatrix{
G_K\times_{\gen{a,l_1,\dots,l_r}}E(\Z\times F_r)\ar^\phi[rrr]\ar^{id_{G_K}\times\psi}[d] &&& EG_K\ar[d]\\
G_K\times_{\gen{a,l_1,\dots,l_r}}E_{\all\gen{a}}(\Z\times F_r)\ar[rrr]&&&E_\M G_K.
}
\end{equation*}
Let $\mathcal{C}(F_r)$ be the geometric realization of the Cayley graph of $F_r$ with respect to the generating set made of the generators of $F_r$ and their inverses. A model for the top-left corner is the disjoint union of copies of $\R\times \mathcal{C}(F_r)$ indexed by $G_K/\gen{a,l_1,\dots,l_n}$. A model for the bottom-left corner is the disjoint union of copies of the mapping cylinder of the projection $\R\times \mathcal{C}(F_r)\rightarrow \mathcal{C}(F_r)$, again indexed by $G_K/\gen{a,l_1,\dots,l_n}$.

\subsection{The Hopf link}\label{ssec:hopf}
With the same machinery we can also build a classifying space for the family of the meridans of the group $G_K=\gen{a,b\mid ab=ba}\simeq\Z^2$ of the Hopf link.
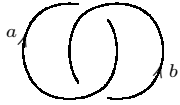
\[\knotholesize{4mm}
\xygraph{
!{0;/r1.5pc/:}
!{\vunder}
!{\vunder-}
[uur]!{\hcap[2]<><{b}}
[l]!{\hcap[-2]<<<{a}}
}
\]

In this link, the meridian of each component serves as the longitude of the other. Hence the classifying space is the double mapping cylinder $Cyl(\R\leftarrow \R^2\rightarrow\R)$ of the projections onto the two components of $EG_K=\R^2$. Here the copy of $\R$ which is the image of the first projection carries the translation action of $b$ and the trivial action of $a$ (i.e., it is $EG_K^{\gen a}$), and conversely for the other copy (which is then $EG_K^{\gen b})$.

\section{Coverings and field extensions}\label{sec:coverings}
We recall a property of prime ideals in Dedekind rings (cf.~e.g.~\cite[\S1.7]{lang_ANT}).
Let $A$ be a Dedekind ring, $K$ its quotient field, $L/K$ a finite separable field extension and $B$ the integral closure of $A$ in $L$. If $P$ is a prime ideal of $A$, then $PB$ is an ideal of $B$ and has a factorization (unique up to permutations)
\[
PB=Q_1^{e_1}\cdots Q_r^{e_r}
\]
into finitely many prime ideals of $B$. The $Q_i$s are precisely the prime ideals of $B$ that lie above $P$ meaning that $Q_i\cap A=P$ (notation: $Q_i|P$).

The natural number $e_i$ is called the \emph{ramification index} of $Q_i$ over $P$. Moreover, for each $i$, let $f_i$ be the (finite) degree of $B/Q_i$ as a field extension of $A/P$ (recall that nonzero prime ideals of Dedekind rings are maximal, hence $B/Q_i$ and $A/P$ are fields).

One can show that
\[
[L:K]=\sum_{i=1}^re_if_i.
\]
If $L/K$ is Galois, then all the $Q_i$s are conjugate to each other, hence all the ramification indices (resp. the residue class degrees) are equal to the same number $e$ (resp. $f$). Thus, in particular, the preceding formula becomes
\begin{equation}\label{eq:efr=degree}
[L:K]=efr.
\end{equation}

A similar formula holds for finite regular coverings of $S^3\setminus V$, where $V$ is a tubular neighbourhood of a knot $K\subseteq S^3$ (cf.~\cite[Chapter 5]{morishita}).
In fact, let $G_K$ be the knot group of $K$ and $H=\gen{a,l}\leq G_K$ be the peripheral subgroup containing the meridian $a$. Consider a normal subgroup $U\trianglelefteq G_K$ of finite index. Then $UH$ is a subgroup of $G_K$ and
\begin{equation*}
\begin{split}
|G_K\colon U|&=|G_K\colon UH|\cdot|UH\colon U|=|G_K\colon UH|\cdot|H\colon H\cap U|=\\
&=|G_K\colon UH|\cdot|H\colon\gen{a}(H\cap U)|\cdot|\gen{a}(H\cap U)\colon H\cap U|\\
&=|G_K\colon UH|\cdot|H\colon\gen{a}(H\cap U)|\cdot|\gen{a}\colon \gen{a}\cap U|.
\end{split}
\end{equation*}

In topological terms, $U$ defines the finite regular covering $\pi^U\colon EG_K/U\to BG_K$ of the closed knot complement $BG_K$. The $G_K$-action on $EG_K$ induces a transitive permutation of the connected components of $\partial EG_K$, which are planes, and the peripheral subgroup $H$ is the (setwise) stabiliser of one such component $P_0$. Hence the subgroup $UH$ is the (setwise) stabiliser of the torus $T_0=\pi^U(P_0)$, which is a connected component of $\partial EG_K/U$. It follows that the cosets in $G_K/UH$ parametrise the connected components of the fibre $(\pi^U)^{-1}(\partial BU)$. So $|G_K\colon UH|$ is analogous to $r$ in Equation \eqref{eq:efr=degree}. In the same spirit, $|\gen{a}\colon \gen{a}\cap U|$ is the least positive integer $k$ such that $a^k\in U$, hence it is analogous to $e$ in Equation \eqref{eq:efr=degree}. Finally, $|H\colon\gen{a}(H\cap U)|$ is analogous to $f$.

The above reasoning also holds for $K$ a link with components $K_1,\dots,K_c$, provided the peripheral subgroups of each component are nondegenerate, that is, isomorphic to $\Z^2$. Such a link will be called \emph{peripherally nontrivial}.
In fact, for $i=1,\dots,c$, let $V_i$ be a tubular neighbourhood of $K_i$, disjoint from the others, and let $H_i=\gen{a_i,l_i}$ be the peripheral subgroup of $G_K$ generated by a meridian $a_i$ and the corresponding longitude $l_i$ of the component $K_i$. Then $G_K$ permutes the connected components of $\partial EG_K/U$ in such a way that two of them are in the same orbit if and only if they are relative to the same component $K_i$ of $K$, that is, if and only if they project to $\partial V_i$.
The stabiliser of a connected component of $\partial EG_K/U$ relative to $K_i$ is $UH_i$. Thus, the connected components of the fibre $(\pi^U)^{-1}(\partial V_i)$ are indexed by $G_K/UH_i$ and it still holds that
\[
\xymatrix@C-8pt@R-10pt{
[G_K\colon U]&=&[G_K\colon UH_i]\ar@{~>}[d]&[H_i\colon\gen{a_i}(H_i\cap U)]\ar@{~>}[d]&[\gen{a_i}\colon \gen{a_i}\cap U]\ar@{~>}[d]\\
&& \vspace{8pt} r_i&f_i & e_i.
}
\]
In other words, the components $K_1,\dots,K_c$ are the analogues of the distinct primes in the ground field of an algebraic number field extension.

It is possible to push this comparison a little forward: in number theory the following holds.
\begin{prop}[\cite{neukirch}, Chapter VI, Corollary 3.8]\label{prop:cor3.8neukirch}
Let $L/F$ be a finite extension of algebraic number fields. If almost all primes of $F$ split completely in $L$, then $L=F$.
\end{prop}

The next result can be seen as a partial analogue in knot theory.
\begin{prop}\label{prop:analogue of cor3.8}
Let $G_K$ be the group of a peripherally nontrivial link $K$ with components $K_1,...,K_c$ and relative tubular neighbourhoods $V_1,\dots,V_c$. Let $U\id G_K$ a normal subgroup of finite index and $\pi^U\colon EG_K/U\to BG_K$ the finite covering induced by $U$. Then $G_K=U$ if and only if $\forall i=1,\dots,c:|\pi_0((\pi^U)^{-1}(\partial V_i))|=|G_K:U|$.
\end{prop}
\begin{proof}
For $i=1,\dots,c$, let $H_i=\gen{a_i,l_i}$ be the peripheral subgroup of $G_K$ generated by the meridian $a_i$ and the longitude $l_i$ of the component $K_i$. It has already been said that $|\pi_0((\pi^U)^{-1}(\partial V_i))|=|G_K:UH_i|$. On the other hand, $|G_K:U|=|G_K:UH_i|\cdot|UH_i:U|$. Hence the equality
\begin{equation}\label{eq:blaaaaah}
 [G_K:U]=|\pi_0((\pi^U)^{-1}(\partial V_i))|
\end{equation}
forces $|H_i:U\cap H_i|=|UH_i:U|=1$, that is, $H_i\subseteq U$. In particular, $a_i\in U$, along with all its conjugates, by normality of $U$. But the generators of $G_K$ are all conjugates to some $a_i$. So, if \eqref{eq:blaaaaah} holds for all $i=1,\dots,c$, then $G_K\subseteq U$. The reverse implication is obvious.
\end{proof}

\subsection{The proof of Proposition B}\label{ssec:propB}
For the remainder of the section, we restrict our attention to prime knots. Let $G_K$ be a prime knot group and $U\id G_K$ be a normal subgroup of finite index. Choose a meridian $a$ and let $H=\gen{a,l}\cong\Z^2$ be the corresponding peripheral subgroup. We want to describe the space $E_\M G_K/U$. By Diagram \eqref{eq:prime diagram explicit}, $E_\M G_K$ is the union $EG_K\cup_{\partial EG_K} (\bigsqcup_{G_K/H}Cyl(\pi\colon\R^2\to\R))$.
The quotient space $EG_K/U$ has finitely many path-connected boundary components $T_1,\dots,T_r$, each homeomorphic to a torus. It remains to understand how each of the extra bits $Cyl(\pi\colon\R^2\to\R)$ is transformed by quotienting over the $U$-action.
$U\cap H$ is a finite-index sublattice of $H$, hence it contains a nontrivial power of $a$. In fact, one can always choose a $\Z$-basis whose first vector is a power of $a$, as follows. Let $\{v,w\}$ be an arbitrary $\Z$-basis of $U\cap H$ and let $e=\min\{n\in\N\setminus\{0\}\mid a^n\in\spn_\Z\{v,w\}\}$. Then, expressing $a^e$ as a $\Z$-linear combination $\alpha v+\beta w$, necesarily $gcd(\alpha, \beta)=1$. As a consequence of Euclid's Algorithm (cf.~\cite[Book 7, Proposition 1]{euclid}), there are $\gamma,\delta\in\Z$ such that 
\[
\mathrm{det}\begin{pmatrix} \alpha & \gamma \\ \beta & \delta \end{pmatrix}=\alpha\delta-\beta\gamma=1.
\]
The condition on the determinant says that $a^e$ and $u=v^\gamma w^\delta$ form a $\Z$-basis for $U\cap H$, that is, $U\cap H=\gen{a^e}\times\gen{u}$.
Now take the extra component $Cyl(\pi\colon EH\to EH^{\gen a})$ corresponding to $H$ (recall that the extra components, as well as the boundary components of $EG_K$, are indexed by the cosets in $G_K/H$). Quotienting first by $\gen{a^e}$ gives an infinite solid cylinder with axis $EH^{\gen a}$ and meridian $a^e$. Quotienting again by $\gen u$ provides a solid torus. The fact that in general $u$ is not orthogonal to $a$ is reflected in the configuration of the lines on the surface of this torus. In particular, if $u=a^pl^q$, the ``old'' longitude $l$ wraps around the axis of the torus with a constant slope of $-p/e$ (and one needs $q$ old longitudes to run a whole lap around the torus).

The quotient map $\pi_U:EG_K\to EG_K/U$ identifies several boundary components of $EG_K$ (one of which, say $P_0$, corresponds to $H$) to a singular boundary component of $EG_K/U$. The preceding argument accounts for a solid torus attached to the boundary component of $EG_K/U$ which is the image of $P_0$.
Clearly the same reasoning works for the other boundary components of $EG_K/U$.
Summing up, the space $E_\M G_K/U$ is obtained by gluing a closed solid torus $V_k\simeq Cyl(\pi\colon\R^2\to\R)/U$ to each boundary component of $EG_K/U$, in such a way that $T_k=\partial V_k$ ($k=1,\dots,r$).

It follows that the fundamental group of $E_\M G_K/U$ can be computed via iterated applications of Seifert and Van Kampen's Theorem. Suppose a presentation of $U$ is given, with set of generators $gen(U)$ and set of relators $rel(U)$.
Let $\pi_1(T_k)=\gen{m_k,l_k\mid m_kl_km_k^{-1}l_k^{-1}}\simeq \Z^2$ and $\pi_1(V_k)=\gen{\ell_k\mid\; }\simeq \Z$. Finally, let $i_k\colon\pi_1(T_k)\to\pi_1(BU)$ and $j_k\colon\pi_1(T_k)\to\pi_1(V_k)$ be the group homomorphisms induced by the inclusion maps. Then,
\[
\pi_1(E_\M G_K/U)= (\dots(\pi_1(BU)\ast_{\pi_1(T_1)}\pi_1(V_1))\ast\dots) \ast_{\pi_1(T_r)}\pi_1(V_r)
\]
(where one has obviously to consider small open neighbourhoods of $BU$ and of each $V_k$ that deformation-retract onto $BU$ and $V_k$ respectively, as well as their intersection deformation-retracts onto $T_k$; this is always possible thanks to tameness).
Hence, $\pi_1(E_\M G_K/U)$ admits the presentation
\[
\begin{aligned}
&\gen{gen(U),\ell_1,\dots,\ell_r\mid rel(U),\; i_k(m_k)j_k^{-1}(m_k),\; i_k(l_k)j_k^{-1}(l_k),\; k=1,\dots,r}\\
=&\gen{gen(U),\ell_1,\dots,\ell_r\mid rel(U),\; i_k(m_k), \;\ell_kj_k^{-1}(l_k),\; k=1,\dots,r}\\
=&\gen{gen(U)\mid rel(U),\; i_k(m_k),\; k=1,\dots,r}
\end{aligned}
\]
(the last equality follows from Tietze's Theorem on equivalence of presentations, cf.~\cite[\S IV.3]{crofox}).
Therefore, there is a short exact sequence of groups
\begin{equation}\label{eq:ses of groups}
\xymatrix{
1\ar[r]&\gen{\gen{i_1(m_1),\dots,i_r(m_r)}}\ar[r]&U\ar[r]&\pi_1(E_\M G_K/U)\ar[r]&1
}
\end{equation}
The elements $i_k(m_k)$ are representatives of the $U$-conjugacy classes of the powers of generators of $G_K$ that lie in $U$. We will denote 
\[
 M_U=\gen{\gen{i_1(m_1),\dots,i_r(m_r)}}
\]
the normal subgroup generated by the $i_k(m_k)$'s.

\subsection{Consequences of Proposition B}\label{sec:consequences}
Given a prime knot group $G_K$, every normal subgroup of finite index $U\id G_K$ induces a finite regular unbranched covering $\pi^U\colon EG_K/U\to EG_K/G_K=BG_K$ and an associated branched covering $\widehat{\pi^U}\colon E_\M G_K/U\to E_\M G_K/G_K=S^3$. Both maps are the quotient by the $G_K/U$-action induced from $G_K$ on the respective total spaces. We digress to collect some properties of these total spaces, keeping the notation of Section \ref{ssec:propB}.

By the construction of $E_\M G_K$ (cf.~Diagram \eqref{eq:prime diagram explicit}),
\begin{equation}\label{eq:branched included unbranched}
EG_K/U\subseteq E_\M G_K/U.
\end{equation}

\begin{lem}\label{lem:proper disc action}
The group $U/M_U$ acts on $E_\M G_K/M_U$ with a covering space action (in the sense of \cite[Page 72]{hatcher}). The quotient space is $E_\M G_K/U$.
\end{lem}
\begin{proof}
The action is induced by the $G_K$-action on $E_\M G_K$, as follows:
\[
uM_U\bullet M_Ux=M_Uux
\]
for $u\in U, x\in E_\M G_K$.
Denote $\mathrm{Fix}_\M G_K=\sqcup_{C\in\M}(E_\M G_K)^C$ and $\overline x=M_Ux\in E_{\M}G_K/M_U$ for $x\in E_{\M}G_K$. If $\stab_{G_K}(x)=1$, i.e., $x\!\in\! E_{\M}{G_K}\setminus(\mathrm{Fix}_\M {G_K})$, then by the construction of $E_\M {G_K}$ we can enlarge $E{G_K}\subseteq E_\M {G_K}$ to a ${G_K}$-CW-complex $Y\subseteq E_\M {G_K}$ which is ${G_K}$-homeomorphic to $E{G_K}$ and contains $x$ in its interior. This amounts to cutting away a small open neighbourhood of $\mathrm{Fix}_\M {G_K}$. Hence it is sufficient to consider only the two cases, $x\in E{G_K}$ or $\stab_{G_K}(x)\neq 1$.

Suppose by contradiction that there exists $\overline x=M_Ux\in E_{\M}{G_K}/M_U$ such that, denoting with $B$ a generic neighbourhood of $\overline x$,
\begin{equation}\label{eq:not prop disc}
\forall B : \exists u\in U\setminus M_U: B\cap uM_UB\neq\emptyset.
\end{equation}
Letting $\widehat{\pi_{M_U}}\colon E_\M {G_K}\to E_\M {G_K}/M_U$ be the canonical quotient map, one has $\widehat{\pi_{M_U}}^{-1}(B)$ $=\bigcup_{m\in M_U}mD$ and $\widehat{\pi_{M_U}}^{-1}(uM_UB)=\bigcup_{m\in M_U}muD$, for a suitable neighbourhood $D$ of $x$.

In the case $x\in E{G_K}$, the neighbourhood $B$ can be chosen so small that $D$ be disjoint from all its translates $uD, u\in U\setminus\{1\}$ because the ${G_K}$-action on $E{G_K}$ is a covering space action by definition. This contradicts \eqref{eq:not prop disc}.
As regards the case $\stab_{G_K}(x)\neq 1$, by construction the stabilizer is infinite cyclic and generated by a meridian, say $\stab_{G_K}(x)=\gen{a}$. The obstruction in order to apply the same reasoning as before is that $D$ cannot be chosen to be disjoint from all of its translates. But, as long as $B$ is small enough, the only problematic translates of $D$ are the translates by powers of $a$, which are identified with $1$ when passing to the quotient by $M_U$.
The claim on the quotient space is obvious.
\end{proof}
\begin{prop}
\label{prop:easter beer}
The quotient map $\widehat{\pi^{M_U}_U}\colon E_{\M}{G_K}/M_U\to E_{\M}{G_K}/U$ associated with the action of Lemma \ref{lem:proper disc action} is the universal covering map of $E_\M {G_K}/U$.
\end{prop}
\begin{proof}
The group $U$ is finitely generated, being a finite-index subgroup of the knot group ${G_K}$. Hence $U/M_U$ is finitely generated too. Moreover, $U/M_U$ is residually finite, being the fundamental group of the compact $3$-manifold $E_\M {G_K}/U$ (cf.~Proposition B and \cite{3mfdgp}). Thus, $U/M_U$ is Hopfian (cf.~e.g.~\cite[\S~4.1]{neumann}), i.e., it is not isomorphic to any of its proper subgroups.

But $U/M_U$ is isomorphic to the group $deck(\widehat{\pi^{M_U}_U})$ of deck transformations of the covering $\widehat{\pi^{M_U}_U}\colon E_{\M}{G_K}/M_U\to E_{\M}{G_K}/U$ (Lemma \ref{lem:proper disc action} and \cite[Proposition 1.40]{hatcher}). On the other hand,
\[
deck(\widehat{\pi^{M_U}_U})\cong\frac{\pi_1(E_\M {G_K}/U)}{(\widehat{\pi^{M_U}_U})_\ast\left(\pi_1(E_\M {G_K}/M_U)\right)},
\]
where $(\widehat{\pi^{M_U}_U})_\ast$ is the homomorphism induced by $\widehat{\pi^{M_U}_U}$ on fundamental groups (cf.~\cite[Proposition 1.39]{hatcher}).

Together with Proposition B, this means
\[
{U}/{M_U}\cong \frac{{U}/{M_U}}{(\widehat{\pi^{M_U}_U})_\ast\left(\pi_1(E_\M {G_K}/M_U)\right)}.
\]
Then, by Hopf property, $(\widehat{\pi^{M_U}_U})_\ast\left(\pi_1(E_\M {G_K}/M_U)\right)=1$, whence the simple connectedness of $E_\M {G_K}/M_U$.
\end{proof}

\begin{prop}
Let ${G_K}$ be a prime knot group and let $U\id {G_K}$ be a normal subgroup of finite index. Then the following are equivalent.
\begin{enumerate}
\item $U=M_U$;
\item The canonical projection $\widehat{\pi^{M_U}_U}\colon E_\M {G_K}/M_U\to E_\M {G_K}/U$ is a trivial covering;
\item The canonical projection ${\pi^{M_U}_U}\colon E {G_K}/M_U\to E {G_K}/U$ is a trivial covering;
\item $\pi_1(E_\M {G_K}/U)=1$;
\item $E_\M {G_K}/U\cong S^3$;
\end{enumerate}
\end{prop}
\begin{proof}
Clearly, (1)$\Rightarrow$(2). The implication (2)$\Rightarrow$(3) holds by restriction, in view of the inclusion \eqref{eq:branched included unbranched}. Moreover, (3)$\Rightarrow$(1) by the Galois correspondence between coverings of $B{G_K}$ and subgroups of ${G_K}$ (cf.~\cite[Theorem 1.38]{hatcher}). The implication (1)$\Leftrightarrow$(4) is given by the exact sequence \eqref{eq:ses of groups} (but it also follows from Proposition \ref{prop:easter beer}). Finally, (4)$\Leftrightarrow$(5) is Poincar\'e Conjecture (now Hamilton-Perel'man's Theorem, as G.~Perel'man himself would probably like to call it).
\end{proof}
\begin{ex}[finite cyclic coverings and Fox completions]
Let $K\subseteq S^3$ be a knot, with tubular neighbourhood $V$, let $X=S^3\setminus V$ and let $G=G_K$ be the knot group. The abelianisation $G/[G,G]$ is infinite cyclic generated by the class of a meridian $a\in G$. The winding number isomorphism (cf.~\cite[Example 2.12]{morishita} $W\colon G/[G,G]\to\Z$ maps that class to $1$. Let $\mathrm{lk(\slot,\slot)}$ be the \emph{linking number} of two curves in $S^3$ (cf.~\cite[\S 1.2]{hillman} or \cite[\S 1.1, \S 4.1]{morishita}). The kernel $U_n$ of the composition
\[
\xymatrix@C+12pt{
G\ar^-{\mathrm{lk(\slot,K)}}[r] & \Z\ar^-{\mathrm{mod}\, n}[r] & \Z/n\Z
}
\]
is a normal subgroup of finite index in $G$.
Since 
\[
\frac{G/[G,G]}{U_n/[G,G]}\cong \frac{G}{U_n}\cong \Z/n\Z
\]
we have a description of $U_n$ as the extension
\[
1\to[G,G]\to U_n\to\gen{a^n}\to 1.
\]
with $a$ a meridian and $\gen{a^n}\cong\Z$.

The covering $\rho_n\colon X_n\to X$ associated to $U_n$, i.e., the unique covering of $X$ with group of deck transformations isomorphic to $\Z/n\Z$, is called the \emph{$n$-fold cyclic covering} of $X$ (or of $K$).
By construction, the covering space $X_n$ has only one boundary component. Letting $H=\gen{a,l}$ be the peripheral subgroup associated to the meridian $a$, $U_n\cap H=\gen{a^n}\times\gen l$. Referring to the notation used at the beginning of this section, this means $e=n=|G:U_n|$, $f=r=1$.

Gluing one solid torus to $X_n$ along its meridian $a^n$, another one to $X$ along $a$, and extending in the obvious way the projection map $\rho_n$ to the manifolds so obtained, one gets the \emph{$n$-fold cyclic branched covering} of $K$
\[
\widehat{\rho_n}\colon \widehat{X_n}\to S^3.
\]
Somebody calls this construction the \emph{Fox completion} of $\rho_n$ (cf. \cite[Example 2.14]{morishita}).

Therefore, the total space of the Fox completion of the $n$-fold cyclic covering of $K$ is a model for the space $E_\M G/U_n$. According to Proposition B, its fundamental group is 
\[
\pi_1(E_\M G/U_n)\cong U_n/\gen{\gen{a^n}}.
\]

To be more concrete, let 
\[
G=\gen{a,b,c\mid ab=bc=ca}=\gen{a,b,c\mid aba=bab}
\]
be the trefoil knot group. Following \cite[Example 2.1]{silwil}, we can describe the commutator subgroup as
$
[G,G]=\gen{a^j(a^{-1}b)a^{-j}\mid j\in\Z}
$
. Using Tietze moves, we can reduce the presentation to
\[
[G,G]=\gen{a^{-1}b,a^{-2}ba}.
\]
This is in accordance with a theorem of Stallings (cf. \cite[\S 5.A]{burzie}), saying that the commutator subgroup of a fibred knot group is free on $2g$ generators, where $g$ is the genus of the knot.

MAGMA Online Calculator
(http://magma.maths.usyd.edu.au/calc/) provided with the code
\begin{verbatim}
G<a,b>:=Group<a,b|a*b*a=b*a*b>;
H:=sub<G|a^-1*b,a^-2*b*a, a^2>;
U_2:=H^G;                       // normal closure
U2:=Rewrite(G,U_2);             // for a better presentation
K:=sub<G|a^2>;
M_2:=K^G;
MU2:=Rewrite(G,M_2);
Q<x,y,z>:=U2/MU2;
Q;
\end{verbatim}
describes $Q=U_2/\gen{\gen{a^2}}$ as a cyclic group of order $3$.

Summing up, we have obtained a regular branched covering of $S^3$, branched over the trefoil knot, with fundamental group $\pi_1(E_\M G/U_2)\cong \Z/3\Z$. The analogous code with the obvious modifications also gives (after playing a little with derived series)
\begin{eqnarray*}
\pi_1(E_\M G/U_3)&\cong & Q_8;\\
\pi_1(E_\M G/U_4)&\cong & SL(2,3);\\
\pi_1(E_\M G/U_5)&\cong & SL(2,5).
\end{eqnarray*}
Here $Q_8$ denotes the group of the units of the quaternions.
For $n=6$, the code gives a runtime error in computing $M_6$, saying that the coset table is not closed.
\end{ex}

\subsection{A Tits alternative for $M_U$}
It is quite difficult to give an explicit description of $M_U$. However, there are some restrictions on its structure. Since knot groups are coherent (as a consequence of Scott's Theorem, cf.~\cite{scott}), $M_U$ is either infinitely generated or finitely presented. In the former case, the index $|U:M_U|$ has clearly to be infinite. In the latter case, one can obtain the following Tits-like alternative: either the index $|U:M_U|$ is finite or $M_U$ is free. In fact, the groups of nontrivial knots have cohomological dimension $2$. Since $U$ has finite index by hypothesis, it has cohomological dimension $2$ as well (cf.~\cite[\S VIII.2]{brown}). If $|U:M_U|$ is finite, then also $M_U$ has cohomological dimension $2$, hence it is not free. On the other side, if $|U:M_U|=\infty$, then by \cite[Corollary 2]{bieri1978} $M_U$ has cohomological dimension $1$, hence it is free by Stallings's Theorem (cf.~\cite{stallings}).

\section{Dualities}\label{sec:duality}
We return to the general setting of Section \ref{sec:construction}. Let $K\subseteq S^3$ be any nontrivial knot, with knot group $G_K$ and a tubular neighbourhood $V_K$. Set $X_K=S^3\setminus V_K$. 
Let $a\in G_K$ be a meridian, $l$ the associated longitude and $H=\gen{a,l}$ the corresponding peripheral subgroup.

\subsection{The proof of Proposition C}
If $U\id G_K$ is a normal subgroup of finite index, then $X_U:=EG_K/U$ is a $3$-manifold with toroidal boundary $T_U$. The number of connected components (tori) of $T_U$ is precisely $G_K/UH$, since $H$ is the (setwise) stabiliser of a boundary component of $EG_K$ (see the beginning of Section \ref{sec:coverings}). The CW-pair $(X_U,T_U)$ produces the long exact sequence in homology
\begin{equation}\label{eq:LES basic}
\xymatrix@R-10pt{
0\ar[r] & H_3(X_U,T_U,\Z)\ar[r] & H_2(T_U,\Z)\ar[r] & H_2(X_U,\Z)\ar[lld] & \\
 & H_2(X_U,T_U,\Z)\ar[r] & H_1(T_U,\Z)\ar[r] & H_1(X_U,\Z)\ar[lld] & \\
 & H_1(X_U,T_U,\Z)\ar[r] & H_0(T_U,\Z)\ar[r] & H_0(X_U,\Z) \ar[r] & 0.
}
\end{equation}

Now Poincar\'e-Lefschetz duality for manifolds with boundary (cf.~\cite[Theorem 3.43]{hatcher} gives isomorphisms $H_i(X_U,T_U,\Z)\cong H^{3-i}(X_U,\Z)$. Moreover, $X_U$ is an Eilenberg-Maclane space $K(U,1)$ for $U$ (cf.~\cite[\S 1.B]{hatcher}), since it is a covering space of $X_K=K(G_K,1)$ (cf.~\S\ref{ssec:preliminaries}). It follows that $H_i(X_U,\Z)\cong H_i(U,\Z)$ and $H^i(X_U,\Z)\cong H^i(U,\Z)$ (cf.~\cite[\S II.4, III.1]{brown}). Summing up, \eqref{eq:LES basic}  the sequence \eqref{eq:poitou-tate U} claimed in Proposition C.

\subsection{The proof of Proposition D}
From now on, coefficients in homology and cohomology are understood to be the integers, unless otherwise specified.

Assume now that $K$ is prime. Recall that $U$ defines the branched covering $\widehat{\pi^U}\colon E_{\M}G_K/U\to E_{\M}G_K/G_K$, as in \S \ref{sec:consequences}. In general $U$ is the group of the link $(\widehat\pi^U)^{-1}(K)$ in the $3$-manifold $E_\M {G_K}/U$. Suppose that $U$ be a knot group, that is, $UH={G_K}$.
Then, applying the abelianisation functor to the sequence \eqref{eq:ses of groups} one gets the exact sequence
\[
\xymatrix{
\gen m\ar^-{\iota}[r] & H_1(U)\ar[r] & H_1(E_\M {G_K}/U)\ar[r] & 0,
}
\]
for $m$ a meridian of $U$. We identify $H\cong\Z^2\cong H^1(H)\cong\gen{m,l}$ and extend $\iota$ to
\begin{eqnarray*}
\widetilde\iota:H_1(H)\to H_1(U),\\
\widetilde\iota(m)=\iota(m),\quad\widetilde\iota(l)=0.
\end{eqnarray*}
This gives an exact sequence
\begin{equation*}
\xymatrix{
H_1(H)\ar^-{\widetilde\iota}[r] & H_1(U)\ar[r] & H_1(E_\M {G_K}/U)\ar[r] & 0.
}
\end{equation*}
We use the Universal Coefficient Theorem for cohomology (cf.~\cite[\S 3.1]{hatcher}), along with torsion-freeness of $0$th homology groups, to write the $\mathrm{hom}$-dual sequence as
\[
\xymatrix{
0\ar[r] & H^1(E_\M {G_K}/U)\ar[r] &  H^1(U)\ar^-{\widetilde\iota^*}[r] & H^1(H)\ar[r] & 0.
}
\]
Note that the identification $H^1(\slot)\to\mathrm{hom}(H_1(\slot),\Z)$ is natural, cf.~\cite[page 196]{hatcher}.

The two previous sequences and the degree $1$ subsequence of \eqref{eq:poitou-tate U} fit into a diagram with exact rows
\begin{equation*}
\xymatrix@R-10pt@C-4pt{
&&& H_1\hskip-0.7pt(\hskip-0.7pt H\hskip-0.7pt)\ar^-{\widetilde\iota}[r]\ar@{=}[d] & H_1\hskip-0.7pt(\hskip-0.7pt U\hskip-0.7pt)\ar@{=}[d]\ar[r] & H_1\!\hskip-1pt\left(\hskip-1.5pt \frac{E_\M {G_K}}{U}\hskip-1.2pt\right)\ar[r] & 0\\
&& H^1\hskip-0.7pt(\hskip-0.7pt U\hskip-0.7pt)\ar@{=}[d]\ar[r] & H_1\hskip-0.7pt(\hskip-0.7pt H\hskip-0.7pt)\ar^-{\alpha}[r] & H_1\hskip-0.7pt(\hskip-0.7pt U\hskip-0.7pt) &&\\
0\ar[r] & H^1\!\hskip-1pt\left(\hskip-1.5pt\frac{E_\M {G_K}}{U}\hskip-1.2pt\right)\ar[r] & H^1\hskip-0.7pt(\hskip-0.7pt U\hskip-0.7pt)\ar^-{\widetilde\iota^*}[r] & H^1\hskip-0.7pt(\hskip-0.7pt H\hskip-0.7pt)\ar^-{\cong}_-{[T_U]\smallfrown\slot}[u] &&&
}
\end{equation*}
in which $[T_U]\smallfrown\slot:H^1(H)\to H_1(H)$ is the Poincar\'e duality isomorphism (cf.~\cite[Theorem 3.30]{hatcher}).
Since the longitude $l\in H$ is null-homologous in $BU$ by hypothesis, $\widetilde\iota$ coincides with the map induced in homology by the inclusion $T_U\hookrightarrow X_U$, that is with $\alpha$. Hence the upper square commutes. Now the  naturality of the maps $H^1(\slot)\to\mathrm{hom}(H_1(\slot),\Z)$ allows to identify $\widetilde\iota^*$ with the map induced on cohomology by the inclusion $T_U\hookrightarrow X_U$. So the commutativity of the lower square follows from the commutativity of the diagram in the proof of Theorem 3.43 in \cite{hatcher}. This establishes Proposition D.


\subsection{The proofs of Theorems E and E$\sharp$}
An arbitrary nontrivial knot group $G_K$ acts cocompactly on the finite-dimensional CW-complexes $\bigsqcup_{G_K/H}\R^2$, $EG_K$ and $\bigsqcup_{G_K/H}Cyl(\pi\colon\R^2\to\R)$, or respectively $G_K\times_{\gen{a,l_1,\dots,l_r}}\R\times \mathcal{C}(F_r)$, $EG_K$ and $G_K\times_{\gen{a,l_1,\dots,l_r}}(Cyl(\pi\colon\R\times \mathcal{C}(F_r)\to\mathcal{C}(F_r))$, appearing in the pushout constructions in Sections \ref{ssec:prime knots} and \ref{ssec:non-prime knots}. Hence the model of $E_\M G_K$ provided by those pushouts is also a cocompact finite-dimensional $G_K$-CW-complex. This is a sufficient condition for $G_K$ to be $\M-FP$ (cf.~\cite[Lemma 3.17]{fluch}, \cite[Theorem 0.1]{luemei}, \cite[\S 3.6]{stjohngreen}).

Now it remains to show that each $M\in\M$ satisfies \eqref{eq:bredon duality}. There are three cases to consider (cf.~\S \ref{ssec:preliminaries}).
\begin{enumerate}
\item If $M=\{1\}$, then $WM=G_K$ and
\[
H^i(G_K,\Z G_K)=
\begin{cases}
\Z^\infty & i=2\\
0 & i\neq 2
\end{cases}
\]
where $\Z^\infty$ means a free abelian group on (countably) infinitely many generators (cf.~\cite{bieck}, in particular Theorem 6.2 and the subsequent discussion). This also shows that $G_K$ cannot be a Bredon Poincar\'e duality group.
\item If $M=\gen{g^{-1}ag}$ for $a$ a generator of $G_K$ and $g\in G_K$, that is if $M$ is maximal in $\M$, then $WM$ is the free group $F_r$, $1\leq r<\infty$, on the classes modulo $M$ of the relevant longitudes. More specifically, if the knot $K$ is prime, then $r=1$ and $WM\cong Z$ is generated by $lM$, where $l$ is the longitude of $K$ corresponding to the meridian $g^{-1}ag$ (cf.~Lemma \ref{lem:non-cable normalizer} and \cite{simon1976}). If $K$ is composite, then $r>1$ and $WM$ is generated by the classes $l_1M,\dots,l_rM$ of the suitable longitudes of the prime factors of $K$ (cf.~\S \ref{ssec:non-prime knots} and \cite{noga1967}). Free groups are duality groups of dimension $1$, and more precisely
\begin{equation}\label{eq:WM free}
H^\bullet(WM,\Z[WM])=H^1(WM,\Z[WM])=\begin{cases}
                                      \Z & r=1\\
                                      \Z^\infty & r>1
                                      \end{cases}
\end{equation}
(cf.~\cite[\S VIII.10, Example 2 and Exercise 2]{brown}). This completes the proof of the case, But we record a consequence of \eqref{eq:WM free} that provides a nice way to prove the remaining case. Namely, $WM$ acts on $\mathcal{C}(F_r)$, the geometric realization of the Cayley graph of $F_r$ with respect to the generating set made of the generators of $F_r$ and their inverses. This action is free, $\mathcal{C}(F_r)$ is contractible and $\mathcal{C}(F_r)/WM$ has finitely many cells, so the cohomology with compact supports of $\mathcal{C}(F_r)$ is isomorphic to $H^\bullet(WM,\Z WM)$:
\begin{equation}\label{eq:H_c(Cay)}
H_c^\bullet(\mathcal{C}(F_r),\Z)=H_c^1(\mathcal{C}(F_r),\Z)=\begin{cases}
                                      \Z & r=1\\
                                      \Z^\infty & r>1
                                      \end{cases}
\end{equation}
(cf.~\cite[Chapter VIII, Proposition 7.5]{brown}).
\item If $M=\gen{g^{-1}a^kg}$ for $a$ a generator of $G_K$, $g\in G_K$ and $k>1$, then $WM=\Z/k\Z\times F_r$. This $WM$ also acts on $\mathcal{C}(F_r)$: one just extends the action of $F_r$ letting elements of $\Z/k\Z$ act as the identity. This action, although not free, has finite isotropy groups, so $H^\bullet(WM,\Z WM)$ is still isomorphic to the cohomology with compact supports of $\mathcal{C}(F_r)$ (cf.~\cite[\S VIII.7, Exercise 4]{brown}). But then, through \eqref{eq:H_c(Cay)}, equation \eqref{eq:WM free} holds equally well also in this case.
\end{enumerate}
This completes the proof of Theorem E.

Now we turn to Theorem E$\sharp$. Let $U\id G_K$ be a normal subgroup of finite index. Then all the $G_K$-CW-complexes appearing in the pushout constructions in Sections \ref{ssec:prime knots} and \ref{ssec:non-prime knots} have a natural $U$-action, obtained restricting the natural $G_K$-action. Since $G_K$ acts cocompactly on each of these $G_K$-CW-complexes and $U$ has finite index, the $U$-action is cocompact as well. Since the isotropy groups of the $U$-action on $E_\M G_K$ are precisely the subgroups in $\M_U$, $U$ is $\M_U-FP$. Moreover, $EG_K/U$ is a finite covering space of the compact manifold $BG_K$, so it is compact as well. Then by \cite[Chapter VIII, Proposition 7.5]{brown}
\[
H^i(U,\Z U)\cong H^i_c(EG_K, \Z)\cong H^i(G_K,\Z G_K)=
\begin{cases}
\Z^\infty & i=2\\
0 & i\neq 2
\end{cases}.
\]
Again, this shows that $U$ cannot be a Bredon Poincar\'e duality group.
Finally, if $M\in\M_U\setminus\{\{1\}\}$, then $W_U(M)=N_U(M)/M=(N_{G_K}(M)\cap U)/M$ is a finite-index subgroup of $N_{G_K}(M)/M=W_{G_K}(M)$, so the argument leading to equation \eqref{eq:WM free} applies again.
\section*{Acknowledgments}
Parts of the results in this article are contained in my PhD thesis. My deepest gratitude goes to  my PhD supervisor, Thomas Weigel, and to my --alas-- unofficial co-supervisor, Brita Nucinkis, for their guidance and exceptional support. Warm thanks to Iain Moffatt for his kind advice and feedback, from which this research benefited in its early stages, and to Ged Corob Cook and Victor Moreno for interesting discussions on classifying spaces.

This research was partially supported by the Italian National Group for Algebraic and Geometric Structures and Their Applications (GNSAGA -- INDAM).
\bibliographystyle{amsplain}
\bibliography{ambclassp}

\end{document}